\documentclass[reqno]{amsart}
\usepackage[dvipsnames*,svgnames]{xcolor}
\usepackage{amsmath,amssymb}
\usepackage{graphicx,subfigure}
\usepackage{tikz}

\newtheorem{theorem}{Theorem}[section]
\newtheorem{thm}{Theorem}[section]

\newtheorem{proposition}[theorem]{Proposition}
\newtheorem{corollary}[theorem]{Corollary}

\newtheorem{openproblem}{Open problem}

\theoremstyle{definition}
 \newtheorem{definition}[theorem]{Definition}

\newtheorem*{definition*}{Definition}

\theoremstyle{remark}
\newtheorem{remark}[theorem]{Remark}

\numberwithin{equation}{section}

\newcommand{\norm}[1]{\left\lVert#1\right\rVert}

\newcommand\minus\backslash

\newcommand\lan\langle
\newcommand\ran\rangle

\renewcommand\leq\leqslant
\renewcommand\geq\geqslant
\newlength{\intwidth}

\newcommand{\Z}{\mathbb{Z}}

\title{Towards a Fluid computer}
\author{Robert Cardona}
\address{Robert Cardona,
Departament de Matemàtiques i Informàtica, Universitat de Barcelona, Gran Via de les Corts Catalanes 585, 08007 Barcelona, Spain \newline \it{e-mail: robert.cardona@ub.edu }
 }
 \author{Eva Miranda}\address{ Eva Miranda,
Laboratory of Geometry and Dynamical Systems $\&$ Institut de Matem\`atiques de la UPC-BarcelonaTech (IMTech),  Universitat Polit\`{e}cnica de Catalunya,  Avinguda del Doctor Mara\~{n}on 44-50, 08028, Barcelona  \\  CRM Centre de Recerca Matem\`{a}tica, Campus de Bellaterra
Edifici C, 08193 Bellaterra, Barcelona
 \it{e-mail: eva.miranda@upc.edu }
 }

\author{Daniel Peralta-Salas} \address{Daniel Peralta-Salas, Instituto de Ciencias Matem\'aticas, Consejo Superior de Investigaciones Cient\'ificas,
28049 Madrid, Spain. \it{e-mail: dperalta@icmat.es} }

 \thanks{Robert Cardona and Eva Miranda were partially supported by the project reference PID2019-103849GB-I00 by MICIU/AEI/10.13039/501100011033/ and the AGAUR grant 2021 SGR 00603.  Eva Miranda is supported by an ICREA Academia Prize 2021 and Bessel Prize of the Alexander Von Humboldt foundation and by the Spanish State Research Agency, through the Severo Ochoa and Mar\'{\i}a de Maeztu Program for Centers and Units
of Excellence in R\&D (project CEX2020-001084-M). Daniel Peralta-Salas is supported by the grants CEX2023-001347-S, RED2022-134301-T and PID2022-136795NB-I00 funded by MCIN/AEI/10.13039/501100011033. All authors are supported by the project Computational,
dynamical and geometrical complexity in fluid dynamics -  AYUDAS FUNDACI\'ON
BBVA A PROYECTOS INVESTIGACI\'ON CIENT\'IFICA 2021.}

\begin{document}
\begin{abstract}

 In 1991, Moore \cite{Mo}  raised a question about whether hydrodynamics is capable of performing computations. Similarly, in 2016, Tao \cite{T0} asked whether a mechanical system, including a fluid flow, can simulate a universal Turing machine. In this expository article, we review the construction in~\cite{CMPP}  of a ``Fluid computer'' in dimension 3 that combines techniques in symbolic dynamics with the connection between steady Euler flows and contact geometry unveiled by Etnyre and Ghrist. In addition, we argue that the metric that renders the vector field Beltrami cannot be critical in the Chern-Hamilton sense~\cite{CH}. We also sketch the completely different construction for the Euclidean metric in $\mathbb{R}^3$ as given in~\cite{CMP2}. These results reveal the existence of undecidable fluid particle paths. We conclude the article with a list of open problems.



\end{abstract}

\maketitle

\section{Introduction}\label{S:intro}

In the realm of decision problems, the halting problem stands as a cornerstone. Turing showed in \cite{Turing} that it is impossible to create a universal algorithm capable of deciding whether any given computer program will stop running or continue indefinitely. This breakthrough marked a significant moment in history of logic and mathematics and led to an important computation model: the Turing machine.

The undecidability of the halting problem raised natural questions: \emph{Which physical models are capable of universal computation and thus exhibit undecidable phenomena?} \emph{What types of physics might transcend computation?} (Penrose~\cite{penrose}) \emph{Can hydrodynamics perform computations?} (Moore~\cite{Mo}). Examples of physical problems where this kind of complexity has been shown to arise include 3D billiards \cite{Mo}, 3D optical systems~\cite{RTY}, neural networks~\cite{S95} or the spectral gap in quantum many-body physics ~\cite{perez}. In a series of works we addressed the following related question posed by Moore \cite{Mo} and revisited by Tao \cite{TNat, T0, T1}: \emph{Is hydrodynamics capable of universal computation?}

Fundamental to the understanding of physical phenomena and dynamical systems in general is the study of the computational complexity that may arise in a given class of systems. This complexity can include undecidable phenomena (implied, for instance, by the capacity of the system to perform universal computation) or computational intractability, which is relevant not only from a purely theoretical point of view but also in terms of applications to developing algorithms to determine the long-term behavior of a given physical system. This complexity is often captured by the existence of undecidable (or computationally intractable) trajectories: there exist explicitly computable initial conditions and open sets of phase space for which determining if the trajectory will intersect that open set can be undecidable or arbitrarily complex from an algorithmic point of view. 

In this survey article, we present some recent results on the connections between Turing machines and the equations of hydrodynamics. In~\cite{CMPP} we constructed stationary solutions of the Euler equations on a Riemannian 3-dimensional sphere that can simulate any Turing machine (i.e., they are Turing complete). In particular, these solutions exhibit undecidable paths. Using different techniques, specifically the Cauchy-Kovalevskaya theorem and the theory of gradient dynamical systems, in~\cite{CMP2} we constructed Turing complete stationary solutions to the Euler equations in Euclidean space. The drawback of these solutions is that they have infinite energy, so they correspond to non-physical fluid flows.

This article is organized as follows: In Section~\ref{S.TM} we introduce the definition of Turing machines and Turing complete dynamical systems. In Section~\ref{S.Euler} we introduce the Euler equations, their stationary solutions, and their connection to contact geometry. The ideas of our constructions of Turing complete Euler flows in~\cite{CMPP} and~\cite{CMP2} are sketched in Sections~\ref{S.Turing} and~\ref{S.introEuler}, respectively. We also highlight a new consequence of the construction in \cite{CMPP}: the Turing completeness of this construction implies that the metric which makes the vector field Beltrami cannot be critical for the Chern-Hamilton energy functional (Proposition \ref{prop:nocritical}). Finally, in Section~\ref{S.open} we conclude with some appealing open problems. 


\section{Undecidability and Turing machines}\label{S.TM}

Undecidable phenomena in mathematics are usually established by reducing a problem to another one which is well-known to be undecidable. We describe here Turing machines and the notion of Turing completeness of a dynamical system, which relates certain properties of the orbits of the system with the halting problem for Turing machines.

\subsection{Turing machines}
A Turing machine is defined as $T=(Q,q_0,q_{halt},\Sigma,\delta)$, where $Q$ is a finite set of states, from which we label an initial state $q_0$ and a (different) halting state $q_{halt}$, a finite set $\Sigma$ called the alphabet of symbols from which we label a blank symbol that we denote by a zero, and a transition function
$$\delta:Q\times \Sigma \longrightarrow Q\times \Sigma \times \{-1,0,1\}.$$
 A configuration of the machine is an element in $Q\times \mathcal{A}$, where $\mathcal{A}$ denotes the countable set of infinite sequences in $\Sigma^\mathbb{Z}$ that have only finitely many non-blank symbols. An input (or starting configuration) of the machine is a configuration of the form $(q_0,t)$, where $q_0$ is the initial state and $t=(t_i)_{i\in \mathbb{Z}}$ is a sequence in $\mathcal{A}$ (referred to as the ``tape" of the machine)

The algorithm follows the rules below. Let $(q,t)\in Q\times \mathcal{A}$ be the configuration at a given step of the algorithm, for instance at the first step it will be an initial configuration.
\begin{enumerate}
 \item[1.] If the current state is $q_{halt}$ then \emph{halt the algorithm} and return $t$ as output.
 \item[2.] Otherwise, compute $\delta(q,t_0)=(q',t_0',\varepsilon)$, where $t_0$ denotes the symbol in position zero of $t$. Let $t'$ be the tape obtained by replacing $t_0$ with $t_0'$ in $t$, and shifting by $\varepsilon$ (by convention $+1$ is a left shift and $-1$ is a right shift). The new configuration is $(q',t')$, and we can go back to step $1$.
\end{enumerate}
The set $\mathcal{P}:=Q\times \mathcal{A}$ is the set of configurations, and the algorithm determines a  global transition function, which is the dynamical interpretation of the machine, given by
$$\Delta: Q\setminus \{q_{halt} \} \times \mathcal{A} \rightarrow \mathcal{P},$$
that sends a configuration to the configuration obtained after applying one step of the algorithm.\\

The so-called halting is (computationally) undecidable, as shown by Alan Turing in $1936$. There is no algorithm that, given any Turing machine $T$ and an initial configuration $c$, can tell in finite time whether $T$ reaches a configuration whose state is the halting state or not. This implies that for some pairs $(T,c)$, the statement ``$T$ halts with $c$" can be true/false but unprovable, i.e., undecidable in the sense of G\"odel.

Another important fact that we will need is the existence of universal Turing machines, which can loosely be defined as follows.
\begin{definition}
A Turing machine $T_U$ is universal if the following property holds. Given any other Turing machine $T$ and an initial configuration $c$ of $T$ there exists a (computable) initial configuration $c_U(T,c)$ of $T_U$, that depends on $T$ and its initial configuration, such that $T$ halts with $c$ if and only if $T_U$ halts with $c_U$.
\end{definition}

As a consequence of the undecidability of the halting problem, a universal Turing machine has inputs for which the halting problem is true/false but unprovable.

\subsection{Turing complete dynamical systems}\label{ss:TC}

In this section we want to formalize what it means for a dynamical system to be able to ``simulate" a universal Turing machine, i.e., to be capable of universal computation. In the next definition the dynamical system $X$ can be discrete or continuous defined on some topological space $M$.

\begin{definition}\label{TC}
A dynamical system $X$ on $M$ is Turing complete if there exists a universal Turing machine $T_U$ such that for each initial configuration $c$ of $T_U$, there exists a (computable) point $p_c \in M$ and a (computable) open set $U_c\subset M$ such that $T_U$ halts with input $c$ if and only if the positive trajectory of $X$ through $p_c$ intersects $U_c$.
\end{definition}

In this case, the halting of a given configuration is equivalent to a certain property of the orbit of $X$ through an explicit point in $M$. It is completely essential that $p_c$ and $U_c$ are in some sense computable. If we are working on a manifold $M$, a point $p_c$ is computable (in terms of $c$) if in some chart the coordinates of $p_c$ can be exactly computed in finite time (in terms of $c$), for instance having explicit rational coordinates. Computability of an open set $U_c$ can be loosely defined as saying that one can explicitly approximate $U_c$ with any given precision. It is clear that a Turing complete dynamical system possesses undecidable trajectories: the long-term behavior of an orbit is undecidable in general.

Even if this is not explicitly stated in the definition of Turing completeness, there is a general approach to constructing them. We first define an ``encoding" function that assigns to each configuration $(q,t)$ of a certain universal Turing machine $T_U$ a set $U_{(q,t)} \subset M$. We map the initial configurations $(q_0,t^{in})$ to points $p_{(q_0,t^{in})}$, and every other configuration as a point or an open set. We then require that for each initial configuration $(q_0,t^{in})$, the trajectory of $X$ through $p_{(q_0,t^{in})}$ sequentially intersects the points or sets corresponding to the configurations obtained by iterating the Turing machine starting with $(q_0,t^{in})$. Namely, the trajectory through $(q_0,t^{in})$ will first intersect the set that encodes $\Delta(q_0,t^{in})$, then the set that encodes $\Delta^2(q_0,t^{in})$ and so on, without intersecting any other coding set in between. With this property, we can consider the open set $U$ obtained as the union of all the open sets $U_{(q_{halt},t)}$ for each $t\in \mathcal{A}$, and the trajectory through $p_{(q_0,t^{in})}$ will intersect $U$ if and only if the machine $T_U$ halts with initial configuration $(q_0,t^{in})$.

\section{Ideal fluids in equilibrium}\label{S.Euler}

The motion of an ideal fluid, which is a fluid that is incompressible and has no viscosity, on a Riemannian manifold $(M,g)$ is described by the Euler equations
\begin{equation*}
\begin{cases}
\frac{\partial}{\partial t} X + \nabla_X X &= -\nabla p\,, \\
\operatorname{div}X=0\,.
\end{cases}
\end{equation*}
Here $p$ is a scalar function called the pressure and $X$ is the velocity field of the fluid. The term $\nabla_X X$ denotes the covariant derivative of $X$ along $X$, and in general, the velocity field $X$ depends on time. A solution to the Euler equations is called stationary whenever $X$ does not depend on time: it models an ideal fluid in equilibrium. Introducing the dual one-form $\alpha=g(X,\cdot)$, the stationary Euler equations can be rewritten as 
\begin{equation*}
\begin{cases}
\iota_Xd\alpha=-dB\,, \\
d\iota_X\mu=0\,,
\end{cases}
\end{equation*}
where $B=p+\frac{1}{2}\norm{X}$ is the Bernoulli function, and $\mu$ is the Riemannian volume form. There is a subclass of stationary solutions, to which our constructions belong, which is the set of Beltrami fields. These are (volume-preserving) vector fields parallel to their curl, which is the only vector field $Y$ such that $\iota_Y\mu=d\alpha$. These correspond to solutions for which the Bernoulli function is constant. \\

Given a Riemannian manifold $(M,g)$, constructing a stationary solution with some desired dynamical property requires finding a solution to the PDE, which might be a rather difficult task. We will do this in the second construction that we describe in Section \ref{S.introEuler}. Another approach, which allows for more flexibility, requires the notion of Eulerisable flow.

\begin{definition}
A vector field $X$ on a manifold $M$ is Eulerisable if there exists an ambient metric $g$ for which $X$ is a stationary solution to the Euler equations in $(M,g)$.
\end{definition}
Hence, if one wants to show that a vector field with specific dynamics models a fluid in equilibrium, there is more room to prove it if we do not fix the metric for which we want the vector field to solve the Euler equations. We will use this approach in the first construction, which we describe in Section \ref{S.Turing}.

There is an important class of vector fields which are Eulerisable: Reeb fields in contact geometry. Recall that a one-form $\alpha$ on a three-dimensional manifold $M$ is a contact form if $\alpha \wedge \alpha >0$ everywhere. The kernel of $\alpha$ is a plane distribution $\xi$ called a contact structure. Any other one-form $f\alpha$ with $f$ a positive function in $M$ is another contact form defining $\xi$ as well. The  Reeb vector field $R$ associated to a given contact form $\alpha$ is determined by the equations
\begin{equation} \label{eq:Reeb}
\iota_R\alpha=1\,, \qquad \iota_Rd\alpha =0\,.
\end{equation}
The flow of $R$ preserves the contact form and thus is a volume-preserving vector field. The same holds for any positive reparametrization $hR$ of $R$, which preserves the volume form $\frac{1}{h}\alpha\wedge d\alpha$. We call a vector field that is a positive multiple of a Reeb field a Reeb-like vector field. The main connection between Reeb-like fields and the Euler equations was established by Etnyre and Ghrist \cite{EG}. For the statement, we recall that a Beltrami field $X$ is rotational if its curl is a non-zero multiple of $X$.

\begin{thm}\label{thm:EGcorres}
Let $M$ be a $3$-dimensional manifold. Any smooth, nonsingular rotational
Beltrami field on $M$ is a Reeb-like field for some contact form on $M$. Conversely, given a
contact form $\alpha$ on $M$ with Reeb field $X$, any nonzero rescaling of $X$ is a smooth, nonsingular
rotational Beltrami field for some Riemannian metric on $M$.
\end{thm}
We will use this theorem in the next section to show that a certain Reeb vector field is a stationary solution to the Euler equations for some metric.

\section{Constructing stationary ideal fluids that are Turing complete (construction 1)}\label{S.Turing}

In this section we review the construction of a Turing complete stationary Euler flow on a Riemannian three-sphere \cite{CMPP}. In the first subsection we introduce Moore's theory of generalized shifts, which is used in the second subsection to construct Turing complete diffeomorphisms of the disk and Reeb flows. 

\subsection{Cantor sets, Turing machines and symbolic dynamics}

The simplest and most ubiquitous of fractals is the Cantor set, which can be drawn in one dimension. We can define it iteratively on an interval of a line: to do this, we divide the interval into three parts and remove the central part; then, we do the same with each of the two new resulting intervals, and so on. This is the so called ternary Cantor set, although we shall usually omit the adjective ``ternary'' in what follows.

But what do fractals have to do with Turing machines? It turns out that both are connected through the Cantor set. By its construction, this set is closely related to the development of a number in base 3. In his 1991 theory, Moore uses that development to encode the configurations of a Turing machine in the square Cantor set, the fractal set obtained by considering the product of two Cantor intervals. Thus, each configuration of the Turing machine corresponds to a unique point in the square Cantor set, and the evolution of the machine can be visualized as a movement between points in this set.

The main result of Moore's theory is that there exists a block transformation that maps the square Cantor set onto itself in such a way that every point associated with a Turing machine's configuration is sent to the point associated with the next configuration according to the algorithm. The details of this block transformation depend on the description of the machine, but we can always think of it as moving a finite number of blocks. In summary, Moore shows that the evolution of any algorithm can be understood as a dynamical system defined on the Cantor set.

These ideas introduced by Moore are key to our  proof in \cite{CMPP}. Building on them, we managed to construct a transformation on a disk (filling in the gaps left by Moore's block transformation to avoid discontinuities) that is capable of simulating a universal Turing machine.

The ideas above are embodied in the concept of generalized shift introduced by 
Moore \cite{Mo} in 1991 to  \emph{simulate any Turing machine}.

A generalized shift is a map that acts on the space of infinite sequences on a given finite alphabet. Specifically, let $A$ be an alphabet and, let $S\in A^\mathbb{Z}$ be an infinite sequence. A generalized shift $\phi:A^\Z \rightarrow A^\Z$ is specified by two maps $F$ and $G$ which depend on a finite number of specified sequence positions in $A^\Z$. Denote by $D_F= \{i,...,i+r-1\}$ and $D_G=\{j,...,j+l-1\}$ the sets of positions on which $F$ and $G$ depend, respectively. These functions indeed take a finite number of different values as they depend on a finite number of positions. The function \( G \) changes the sequence exclusively at the positions specified by \( D_G \):

\begin{align*}
G:A^l &\longrightarrow A^l \\
(s_{j}...s_{j+l-1}) &\longmapsto (s_{j}'...s_{j+l-1}')
\end{align*}
Here $s_j...s_{j+l-1}$ are the symbols at the positions $j,...,j+l-1$ of an infinite sequence $S\in A^\mathbb{Z}$.

On the other hand, the function $F$ assigns to the finite subsequence $(s_{i},\ldots,s_{i+r-1})$ of the infinite sequence $S\in A^\mathbb{Z}$ an integer:
\[ F : A^{r} \longrightarrow \mathbb{Z}. \]

The generalized shift $\phi : A^\mathbb{Z} \rightarrow A^\mathbb{Z}$ corresponding to $F$ and $G$ is defined as follows:
\begin{itemize}
\item Compute $F(S)$ and $G(S)$.
\item Modify $S$ by changing the positions in $D_G$ according to the function $G(S)$, obtaining a new sequence $S'$.
\item Shift $S'$ by $F(S)$ positions. In other words, we obtain a new sequence $s''_n = s'_{n+F(S)}$ for all $n\in \mathbb{Z}$.
\end{itemize}
The sequence $S''$ is then $\phi(S)$.
\newline

Given a Turing machine, there is a generalized shift $\phi$ conjugate to it. Conjugation means that there exists an injective map $\varphi: \mathcal{P} \rightarrow A^\mathbb{Z}$ such that the global transition function of the Turing machine is given by $\Delta = \varphi^{-1} \phi \varphi$. In fact, if the Turing machine is reversible, it can be shown that the generalized shift is bijective.

\subsubsection*{Generalized shifts and square Cantor sets}
 Generalized shifts are conjugate to maps of the \emph{square Cantor set} $C^2:=C\times C \subset I^2$, where $C$ is the (standard) \emph{Cantor ternary set} in the unit interval $I=[0,1]$.

 Without loss of generality, we can assume $A=\{0,1\}$. Given $s=(\ldots s_{-1}.s_0s_1\ldots)\in A^\mathbb{Z}$, we can associate to it an explicitly constructible point in the square Cantor set. In order to do this, it is enough to express the coordinates of the assigned point in base $3$. The coordinate $x$ corresponds to the expansion $(x_1,x_2,\ldots)$ in base $3$ where $x_i=0$ if $s_{-i}=0$ and $x_i=2$ if $s_{-i}=1$. Similarly, the coordinate $y$ corresponds to the expansion $(y_0,y_1,\ldots)$ where $y_i=0$ if $s_i=0$ and $y_i=2$ if $s_i=1$.

Moore proved that any generalized shift is conjugate to the restriction on the square Cantor set of a piecewise linear map defined on blocks of the Cantor set in $I^2$. This map consists of finitely many area-preserving linear components. If the generalized shift is bijective, then the image blocks are pairwise disjoint. An example is depicted in Figure \ref{fig:Cantormap}.
 Each linear component is the composition of two linear maps: a \emph{translation} and a positive (or negative) power of the \emph{horseshoe map} (or the Baker's map).

 \begin{figure}[!ht]\label{fig:Cantormap}
\begin{center}
\includegraphics[scale=0.3]{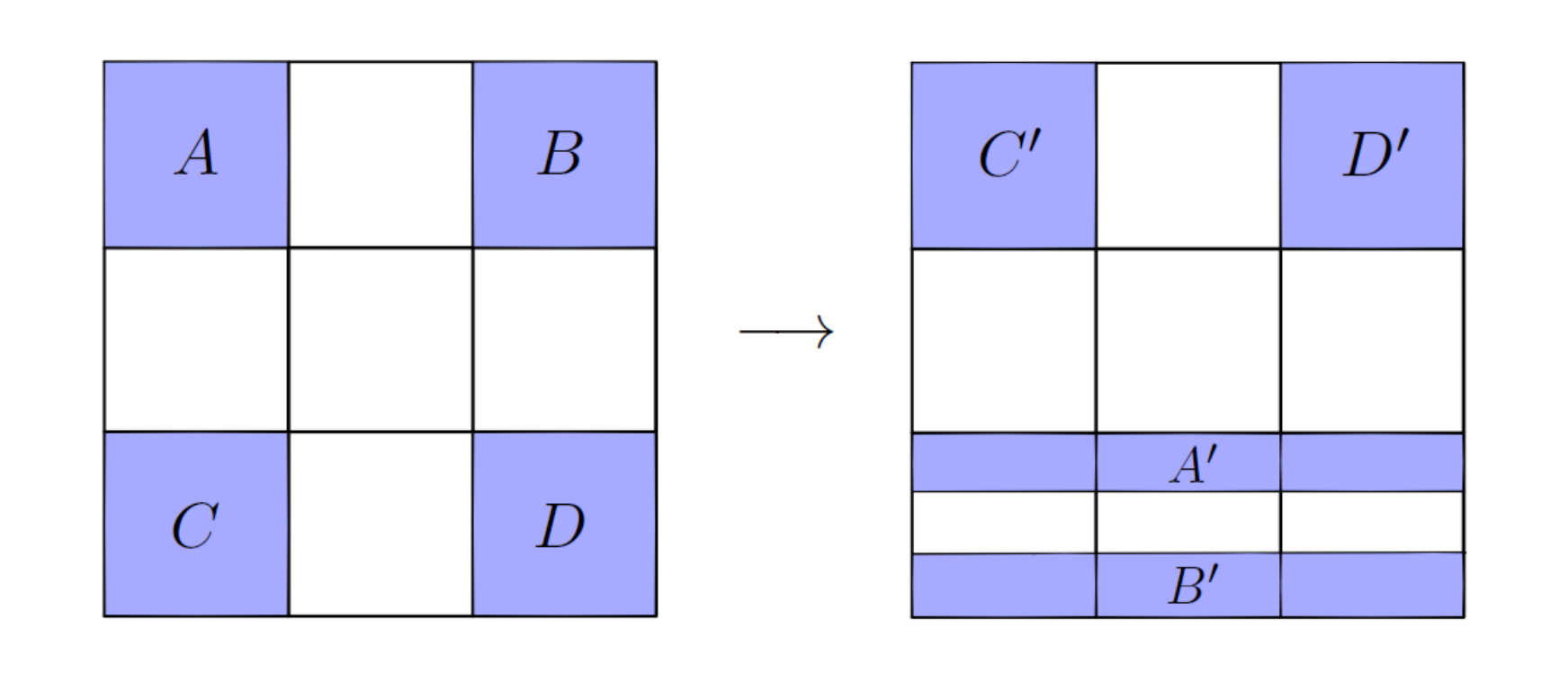}
\caption{Example of a map by blocks of the square Cantor set.}
\end{center}
\end{figure}

The figure above illustrates a mapping by blocks associated with a generalized shift $\phi:A 
\rightarrow A$, where the alphabet is $A=\{0,1\}$ and both $D_F$ and $D_G$ are defined as $\{-1,0\}$. The functions $F$ and $G$ are defined as follows:
$$G(0.1)=0.1,\enspace G(1.1)=0.0,\enspace G(0.0)=0.1,\enspace G(1.0)=1.1,$$ 
and $$F(0.1)=-1,\enspace F(1.1)=-1,\enspace F(0.0)=0,\enspace F(1.0)=0.$$ 
In this context, the Cantor blocks $A$, $B$, $C$, and $D$ represent sequences where positions $-1$ and $0$ respectively contain the pairs $(0.1)$, $(1.1)$, $(0.0)$, and $(1.0)$. Their respective images are denoted as $A'$, $B'$, $C'$, and $D'$.


\subsection{Area-preserving maps, Poincaré sections and Turing complete Reeb flows}

In \cite{CMPP}, and as suggested by Moore \cite{Mo}, we proved that any bijective generalized shift can be extended to an area-preserving diffeomorphism of the disk which is the identity near the boundary.

\begin{proposition}\label{prop:areaShift}
For each bijective generalized shift and its associated map of the square Cantor set $\phi$, there exists a $C^\infty$ area-preserving diffeomorphism of the disk $\varphi:D\rightarrow D$ which is the identity in a neighborhood of $\partial D$ and whose restriction to the square Cantor set is conjugate to $\phi$.
\end{proposition}

This proposition indicates that a Turing-complete bijective generalized shift, derived from a reversible universal Turing machine, leads to the existence of a Turing-complete area-preserving diffeomorphism of the disk as shown in \cite[Theorem 5.2]{CMPP}. Furthermore, we prove that this diffeomorphism can be identified as the first-return map on a disk-like Poincar\'e  section of a Reeb field.

\begin{thm}\label{thm:ReebDisk}
Let $(M,\xi)$ be a contact $3$-manifold, and $\varphi:D \rightarrow D$ be an area-preserving diffeomorphism of the disk which is the identity (in a neighborhood of) the boundary. Then there exists a contact form $\alpha$ whose associated Reeb vector field $R$ exhibits a Poincar\'e section with first return map conjugate to $\varphi$. In particular,  this implies the existence of a Turing-complete Reeb field  $R$ on $(M,\xi)$ in the sense defined in Definition~\ref{TC}.
\end{thm}

The final step involves assembling all the pieces of the puzzle together.
Combining Proposition~\ref{prop:areaShift}, Theorem~\ref{thm:ReebDisk} and Theorem~\ref{thm:EGcorres}, which establishes the correspondence between Beltrami fields and Reeb flows, we achieve the desired result for stationary Euler flows.

\begin{corollary}\label{cor:turing}
There exists an Eulerisable field $X$ on $\mathbb S^3$ that is Turing complete. The metric $g$ that makes $X$ a solution of the stationary Euler equations can be assumed to be the round metric in the complement of an embedded solid torus. Furthermore, $X$ admits a disk-like cross section $D$ whose first return map is a diffeomorphism $\varphi:D\to D$ which is Turing complete and the identity in a neighborhood of $\partial D$.
\end{corollary}

The fact that the metric can be assumed to be the round one in the complement of an embedded solid torus needs some explanation. When applying Theorem~\ref{thm:ReebDisk}, we take as ambient manifold the standard contact sphere $(\mathbb S^3,\xi_{std})$. Then, the contact form whose Reeb field realizes a given area-preserving diffeomorphism of the disk as a Poincar\'e map can be chosen to coincide with the standard contact form $\alpha_{std}$ outside a solid torus. To conclude, one can check that the metric associated to $\alpha$ via Theorem~\ref{thm:EGcorres} can be taken to be the round one whenever $\alpha$ coincides with $\alpha_{std}$.

The metric $g$ in Corollary~\ref{cor:turing} is any metric compatible with the contact form associated to the Reeb field $R\equiv X$ constructed in Theorem~\ref{thm:ReebDisk}. It is natural to ask if one can choose an ``optimal'' metric among all compatible metrics. To the end of defining optimal compatible metrics, Chern and Hamilton~\cite{CH} introduced the functional
\[
E(g)=\int_{\mathbb S^3}|L_Rg|^2_g \text{vol}_g
\]
on the space of compatible metrics, and considered metrics that are (local) minima of this functional, what they called \emph{critical compatible metrics}. The following proposition shows that the metric $g$ in Corollary~\ref{cor:turing} cannot be critical:

\begin{proposition}\label{prop:nocritical}
There is no critical metric $g$ that makes the Reeb field $R$ from Theorem~\ref{thm:ReebDisk} on $M=\mathbb S^3$ a Beltrami field.   
\end{proposition}
\begin{proof}
We describe $\mathbb S^3$ as the set of points $(z_1,z_2)\in\mathbb C^2$ with $|z_1|^2+|z_2|^2=1$. We parametrize $\mathbb S^3$ using Hopf coordinates
\[
(z_1,z_2)=(\cos s e^{i\phi_1},\sin s e^{i\phi_2})
\]
with $s\in[0,\pi/2]$, $\phi_{1,2}\in[0,2\pi)$. 

Assume that the Reeb field $R$, associated to a contact form $\alpha$, admits a critical compatible metric $g$. Then, according to the classification in~\cite{MPS}, the Riemannian contact structure $(\mathbb S^3,\alpha,g)$ is Sasakian. Sasakian structures on $\mathbb S^3$ are very rigid and are classified, in particular Belgun showed~\cite[Lemma~1]{Be} that the Reeb field corresponding to a Sasakian contact structure on $\mathbb S^3$ is orbitally conjugate to a vector field that reads in Hopf coordinates as
\[
l\partial_{\phi_1}+k\partial_{\phi_2}=:R_{lk}
\]
for some positive constants $l,k$. Taking as Poincar\'e section the disk $D:=\{\phi_1=0\}$, we easily compute the first return map to be 
\[
\Pi_{lk}(s,\phi)=(s,\phi+\omega)
\]
with $\omega:=\frac{2\pi k}{l}$. $\Pi_{lk}$ is a rigid rotation of frequency $\omega$. Since the Reeb field $R$ constructed in Corollary~\ref{cor:turing} has a Poincar\'e map on a transverse disk that is the identity in a neighborhood of the boundary, this immediately implies that $\omega=0$ mod $2\pi$, so $\Pi_{lk}=id$. It is obvious that the identity map cannot simulate a universal Turing machine, so $R$ cannot be Turing complete. This concludes the proof of the proposition.  
\end{proof}

\begin{remark}\label{rem:rigidrot}
More generally, it can be argued that a Seifert flow  $l\partial_{\phi_1}+k\partial_{\phi_2}$ (or a rigid rotation on a disk), cannot be Turing complete. Suppose that given two computable sets we can determine whether they intersect (this might not be the case if they are at distance zero, but we assume there is a black box like an oracle machine that can determine this). Even under this assumption, the reachability problem of a Turing complete system is undecidable, i.e., there is an exactly computable initial condition $p$ and a computable open set $U$ for which determining if the positive trajectory of the system starting at $p$ intersects $U$ is an undecidable problem. On the other hand, for a Seifert flow, an exactly computable initial condition belongs to a computable invariant set $V$: an invariant circle if $l/k$ is rational or an invariant torus if it is irrational. The positive trajectory of $p$ intersects $U$ if and only if the intersection of $U$ and $V$ is non-empty, which can be determined using our black box. Thus, under the assumption that we have access to such an oracle machine, the reachability problem is decidable, and we conclude that a Seifert flow (and hence a rigid rotation) cannot be Turing complete.
\end{remark}

\section{Constructing stationary ideal fluids that are Turing complete (construction 2)}\label{S.introEuler}

In the former construction of a Turing complete stationary solution of the Euler equations, the Riemannian metric is deformed and adapted to the underlying contact geometry.
These deformations are controlled using sophisticated techniques, which can make the metric highly intricate in certain regions.
The contact/hydrodynamics duet fails when the metric is prescribed, requiring new strategies. In the standard classical formulation of the Euler equations in 3D, the prescribed metric is the Euclidean one on $\mathbb R^3$.

In this section, we provide an overview of the construction  in \cite{CMP2}, where we achieve a universal Turing machine simulated by a steady fluid flow in Euclidean 3D-space. 

\subsection{Statement of the theorem}
 The main theorem we proved in \cite[Theorem 1]{CMP2} is the existence of Turing complete stationary solutions to the Euler equations in $\mathbb{R}^3$ endowed with the Euclidean metric.

\begin{theorem}
There exists a Turing complete Beltrami field $u$ in Euclidean space $\mathbb{R}^3$.
\end{theorem}

The strategy of the proof can be sketched as follows.

\begin{enumerate}
\item We show that there exists a Turing complete vector field $X$ in the plane $\mathbb{R}^2$ that is of the form $X=\nabla f$ where $f$ is a smooth function.
\item Furthermore, we require that if we perturb $X$ by an error function $\varepsilon:\mathbb{R}^2\rightarrow \mathbb{R}$ that decays rapidly enough at infinity, then we obtain a vector field that is Turing complete as well.
\item We show that a vector field in $\mathbb{R}^2$ of the form $X=\nabla g$, where $g$ is an entire function, can be extended to a Beltrami field $v$ in $\mathbb{R}^3$ such that $v|_{z=0}=X$. That is, $v$ leaves the plane $\{z=0\}$ invariant and coincides with $X$ there.
\item We approximate $f$ by an entire function $F$ with an error that decays rapidly enough. Hence $\tilde X=\nabla F$ is Turing complete and extends as a Beltrami field $u$ on $\mathbb R^3$. It easily follows that $u$ is Turing complete as well.
\end{enumerate}

In the next two subsections, we outline the main steps of the proof. In the first subsection, we provide an overview of the arguments for steps (i) and (ii). The third step utilizes a global Cauchy-Kovalevskaya theorem tailored to the curl operator, which we elaborate on subsequently. The fourth step involves a general result concerning the approximation of smooth functions by entire functions with errors that decay arbitrarily, as detailed in \cite{FG}.

\subsection{Weakly robust Turing complete gradient flow in the plane}

In this subsection, we outline the construction of a Turing complete gradient flow on $\mathbb R^2$ and discuss how to ensure that its Turing completeness is robust under perturbations that decay sufficiently fast at infinity. Following the rough idea presented in Subsection~\ref{ss:TC}, we will first show how to encode the configurations of a given universal Turing machine $T_U$ into $\mathbb{R}^2$. Without loss of generality, we assume that $T_U=(Q,\Sigma, q_0, q_{halt}, \delta)$ with $Q=\{1,...,m\}$ for some $m\in \mathbb{N}$ and $\Sigma=\{0,1\}$. \\

\paragraph{\textbf{The encoding.}} We first construct an injective map from $\mathcal{P}=Q\times \mathcal{A}$ to $I=[0,1]$, where we recall that $\mathcal{A}$ is the set of compactly supported sequences in $\Sigma^\mathbb{Z}$. Given $(q,t)\in \mathcal{P}$, write the tape as
\[
\cdots 000t_{-a}\cdots t_b00\cdots,
\]
where $t_{-a}$ is the first negative position such that $t_{-a}=1$ and $t_b$ is the last positive position such that $t_b=1$. If every symbol in a negative (or positive) position is zero, we choose $a=0$ (or $b=0$ respectively).
Set the non-negative integers given by concatenating the digits $s:=t_{-a}\cdots {t_{-1}}$, $r:=t_b\cdots t_0$, and introduce the map
\[
\varphi(q,t):=\frac{1}{2^q3^r5^s}\in (0,1),
\]
which is injective and its image accumulates at $0$. There exist pairwise disjoint intervals $I_{(q,t)}$ centered at $\varphi(q,t)$, for instance of size $\frac{1}{4}\varphi(q,t)^2$. To introduce an encoding into $\mathbb{R}^2$ we proceed as follows. Fix $\epsilon>0$ small, we encode $(q,t)$ as the open set
\[
U_{(q,t)}:=\bigcup_{j,k=0}^\infty I^j_{(q,t)}\times (k-\epsilon/2,k+\epsilon/2)
\]
where $I^j_{(q,t)}:=I_{(q,t)}+(2j,2j)$. In other words, we are looking at any interval of the form $I^{j,k}:=[2j,2j+1]\times \{k\} \subset \mathbb{R}^2$, with $j,k\in \mathbb{N}$, and considering an $\varepsilon$-thickening of $I_{(q,t)}$ understood as a subset of $I^{j,k}\cong [0,1]$. Figure \ref{fig:1} gives a visualization of part of one of the open sets $U_{(q,t)}$ in a region of the plane.

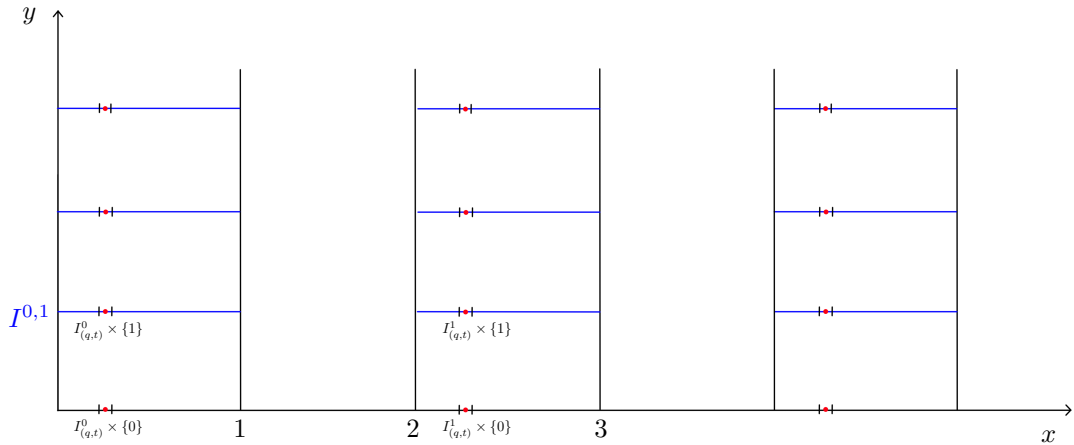
\begin{figure}[!h]
\begin{center}
\begin{tikzpicture}
     \node[anchor=south west,inner sep=0] at (0,0) {\includegraphics[scale=0.12]{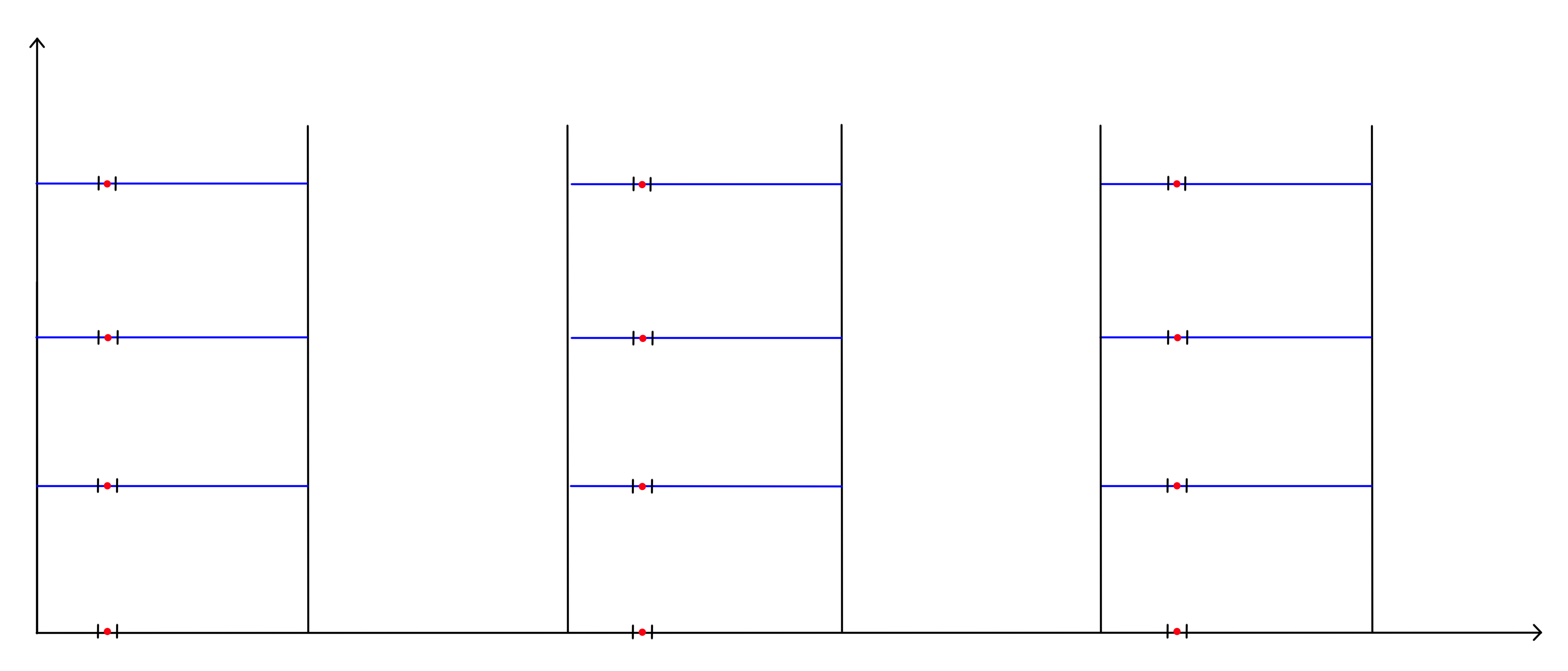}};
   \node at (13.5,0) {$x$};
   \node at (-0.05,5.6) {$y$};
     \node at (2.75,0.1) {$1$};
    \node at (5.05,0.1) {$2$};
    \node at (7.55,0.1) {$3$};
    
    \node[blue] at (-0.05, 1.6) {$I^{0,1}$};
    
    \node[scale=0.55] at (1,0.1) {$I^0_{(q,t)} \times \{0\}$};
    \node[scale=0.55] at (1,1.4) {$I^0_{(q,t)} \times \{1\}$};

    \node[scale=0.55] at (5.9,0.1) {$I^1_{(q,t)}\times \{0\}$};
     \node[scale=0.55] at (5.9,1.4) {$I^1_{(q,t)}\times \{1\}$};   

\end{tikzpicture}
\caption{The open set $U_{(q,t)}$ is the $\varepsilon$-thickening of the intervals $I^j_{(q,t)}\times \{k\}$}
\label{fig:1}
\end{center}
\end{figure}

The countable set of initial configurations $\mathcal{P}_0=\{(q_0,t)\enspace | \enspace t\in \mathcal{A}\}$ admits a (computable) ordering which we will not specify, so that we can write it as
$\mathcal{P}_0=\{c_i=(q_0,t^i) \enspace | \enspace i\in \mathbb{N}\}$. Given $c_i$, the initial condition associated to the vector field that we will construct will be $p_{c_i}=(\varphi(c_i)+i,0) \in \mathbb{R}^2$. This corresponds to the point $\varphi(q_0,t^i)$ of the copy $I^{i,0}$ of the several intervals we considered. \\

\paragraph{\textbf{Integral curves capturing the steps of the algorithm.}} Iterating the global transition function from an initial configuration $c_i=(q_0,t^i)$ gives a countable sequence of configurations $c_i^k=(q_k,t_k^i)=\Delta^k(c_i)$ for each $k$ an integer greater than $1$. On each band $[2i,2i+1]\times [0,\infty)$, we construct a smooth curve $\gamma_i$ such that $\gamma_i\cap \{[2i,2i+1]\times \{k\}\}$ is the point $(2i+\varphi(q_k,t_k^i), k)$, which lies in $U_{(q_k,t_k^i)}$, see Figure \ref{fig:2}.

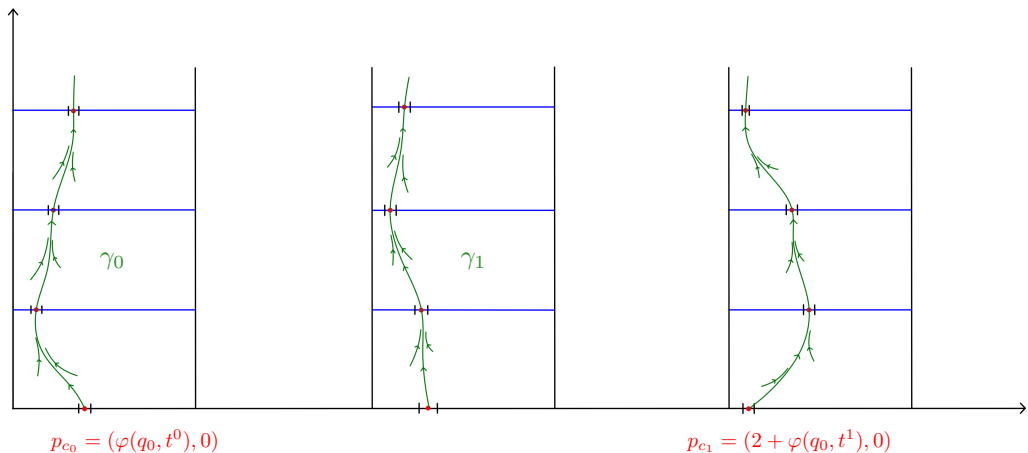
\begin{figure}[!h]
\begin{center}
\begin{tikzpicture}
     \node[anchor=south west,inner sep=0] at (0,0) {\includegraphics[scale=0.12]{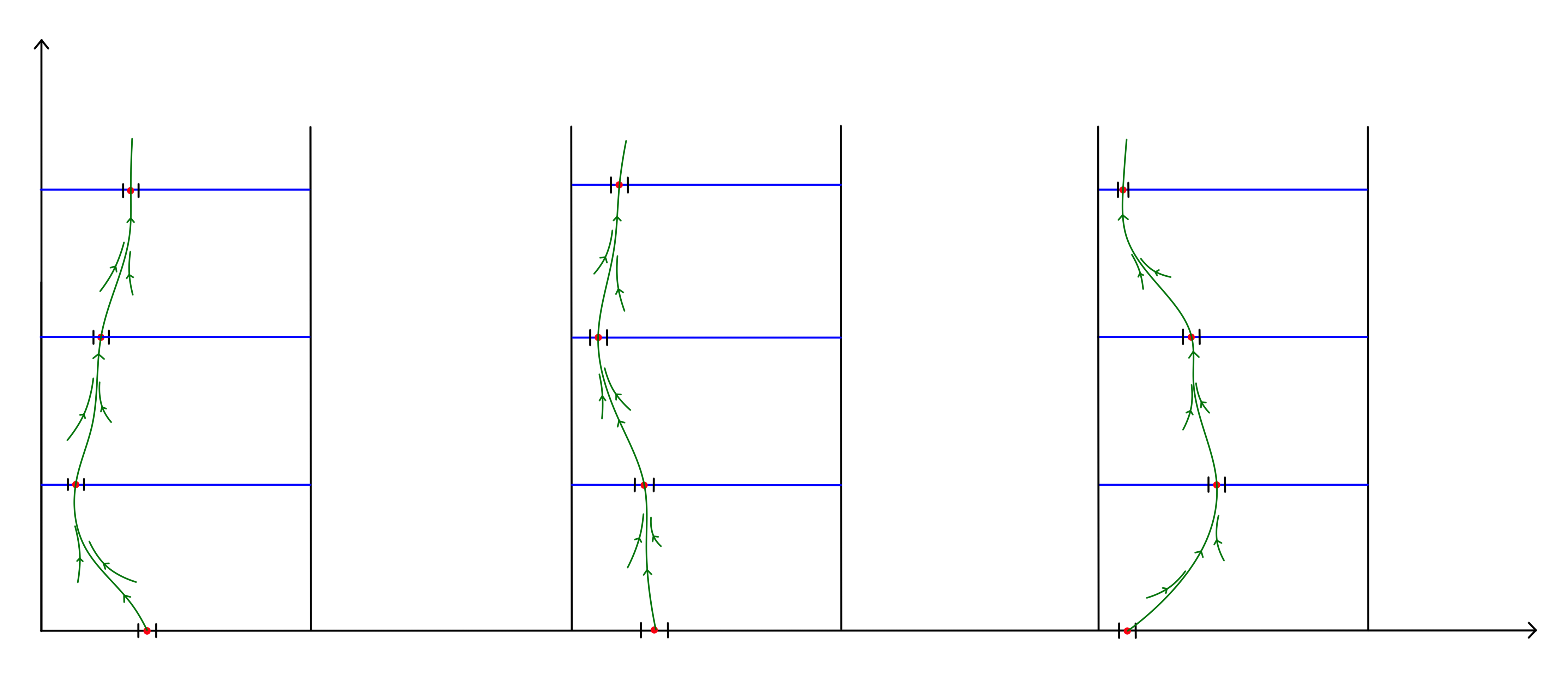}};

          \node[scale=0.8, color=red] at (2, 0.1) {$p_{c_0}=(\varphi(q_0,t^0),0)$};
     \node[scale=0.8, color=red] at (10.7, 0.1) {$p_{c_1}=(2+\varphi(q_0,t^1),0)$};

     \node[color=ForestGreen] at (1.7,2.5) {$\gamma_0$};
      \node[color=ForestGreen] at (6.5,2.5) {$\gamma_1$};

\end{tikzpicture}
\caption{Integral curves following the computations of the Turing machine}
\label{fig:2}
\end{center}
\end{figure}

We conclude by constructing a gradient field $X=\nabla f$ such that each $\gamma_i$ is an integral curve of $X$. Observe now that given an initial condition $c_i$, the integral curve through $p_{c_i}$ will intersect sequentially the open sets $U_{(q_k,t_k^i)}$, thus keeping track of the computations of the machine with initial configuration $c_i$. One easily shows that $X$ is Turing complete, where to each $c_i$ we assign the initial condition $p_{c_i}$, and the open set $U$ for which the trajectory through $c_i$ intersects $U$ if and only if the machine halts with initial configuration $c_i$ is simply $U=\bigcup_{t\in \mathcal{A}} U_{(q_{halt},t)}$, that is, every open set encoding a halting configuration.\\

\paragraph{\textbf{Weak robustness and conclusion.}} Recall that in order to apply the Cauchy-Kovalevskaya theorem, we need $X$ to be the gradient of an entire function. To achieve this, we construct $X$ in a way that the flow normally contracts towards each curve $\gamma_i$ at a strong enough rate, see Figure \ref{fig:2}. This can be used to show that if we perturb $X$ by an error function $\varepsilon(x,y)$ with fast decay at infinity, we obtain a vector field that is again Turing complete. This is because even if the curve $\gamma_i$ will no longer be an integral curve, the integral curve through any of the points $p_{c_i}$ of the perturbed vector field will still intersect sequentially the open sets $U_{(q_k,t_k^i)}$, hence capturing the computations of the Turing machine. The fast decay of the error is necessary since the open sets $U_{(q_k, t_k^i)}$ have no uniform lower bound on their size. This is because the intervals $I_{(q,t)}$ accumulate at zero, and hence their size tends to zero. One can estimate the decay rate of the size of the open sets that need to be intersected by the curves in terms of the distance to the origin, and hence robustness can be achieved under fast decay errors.

The construction concludes by approximating $f$ by an entire function $F$ (using~\cite{FG}), and applying the Cauchy-Kovalevskaya theorem for the curl to the Cauchy datum $\nabla F$ on the plane $\{z=0\}$ (see the following subsection).

\subsection{A Cauchy-Kovalevskaya theorem}
The Cauchy–Kovalevskaya theorem is the main tool to establish local existence and uniqueness for solutions to analytic partial differential equations related to Cauchy data. A Cauchy problem involves finding the solution to a partial differential equation satisfying specific conditions on a non-characteristic hypersurface within the domain. It applies to a wide range of PDEs as far as they admit non-characteristic surfaces.

Unfortunately, the curl operator has no non-characteristic surfaces. The reason is that its Fourier symbol is given by the $3 \times 3$ matrix

$$
\mu(\xi)=i\left(\begin{array}{ccc}
0 & -\xi_{3} & \xi_{2} \\
\xi_{3} & 0 & -\xi_{1} \\
-\xi_{2} & \xi_{1} & 0
\end{array}\right)\,,
$$
which is antisymmetric, so
$\operatorname{det} \mu(\xi)=0$ for all values $(\xi_{1}, \xi_{2},\xi_{3}) \in \mathbb{R}^{3}$.

This means that the Cauchy datum that one can prescribe on a surface cannot be arbitrary, there must be obstructions. Indeed, taking the plane $\{z=0\}$ as the Cauchy surface,
and noticing that the third component of the Beltrami equation is:

$$
\frac{\partial u_{2}}{\partial x}-\frac{\partial u_{1}}{\partial y}=\lambda u_{3}\,,
$$
if we assume that the Cauchy datum is tangent to the surface, that is $u_{3}(x, y, 0)=0$, we easily obtain
$$\frac{\partial u_{2}(x, y, 0)}{\partial x}-\frac{\partial u_{1}(x, y, 0)}{\partial y}=0\,.$$
Therefore, if $\Sigma=\{z=0\}$ and the Cauchy datum is $v(x, y)=v_{1} \partial_{x}+v_{2} \partial_{y}$, a necessary condition for the local solvability of the Beltrami equation is $\frac{\partial v_{1}}{\partial y}-\frac{\partial v_{2}}{\partial x}=0$ on $\mathbb{R}^{2}$, and this happens if and only if
$$v=\nabla_{\mathbb{R}^{2}} F$$ for some analytic function
$$
F: \mathbb{R}^{2} \rightarrow \mathbb{R}
$$

In~\cite{CMP2} we proved that the aforementioned necessary condition is also sufficient for local solvability. Moreover, using a global version of Cauchy-Kovalevskaya due to Pong\'erard and Wagschal~\cite{PW} we also showed that the Beltrami field one obtains using Cauchy-Kovalevskaya is defined on the whole $\mathbb R^3$ provided that the analytic planar function $F$ is entire. This is the content of the following theorem:

\begin{theorem}
Let $F(x, y)$ be an entire function (i.e., it extends as a holomorphic function on $\mathbb C^{2}$ ). Then for any $\lambda \neq 0$ there exists a unique Beltrami vector field $u$ in $\mathbb{R}^{3}$ that solves the Cauchy problem:

$$
\left\{\begin{array}{l}
\operatorname{curl} u=\lambda u \\
\left.u\right|_{z=0}=\nabla_{\mathbb{R}^{2}} F
\end{array}\right.
$$

\end{theorem}
\section{Some open problems}\label{S.open}

We conclude with a list of open problems. The first one is inspired by Etnyre and Ghrist's observation that there is a correspondence between Reeb fields and Beltrami flows, cf. Section~\ref{S.Euler}. This connection allows one to show that the $n$-body problem of celestial mechanics can be interpreted as a steady Euler flow on some Riemannian manifold (which corresponds to positive energy levels)~\cite{josepdanieleva}. While the chaotic nature of several systems in celestial mechanics has been studied by various authors (see, for instance, \cite{BGG}), we are not aware of any research analyzing the computational power of the \( n \)-body problem from the perspective of Turing machines. This leads us to consider the following:


\begin{openproblem} 
Is the $n$-body problem in celestial mechanics Turing complete for some $n>2$ (and some values of masses)?
    
\end{openproblem}
 
The second problem we want to state concerns the connections between dynamical chaos and computational power. In \cite{entropypaper} we proved that a natural class of Turing machines (that we called ``regular''), which include most of known examples of universal Turing machines, have positive topological entropy; however, as explained in~\cite{entropypaper}, there are universal Turing machines with zero topological entropy. More refined versions of entropy in dynamical systems are given by 
entropy-type invariants of subexponential growth (e.g. the polynomial entropy~\cite{kushnirenko,jeanpierremarco,HLR}) that may be non-vanishing even for systems with zero topological entropy. For a recent comprehensive account on the subject see~\cite{Katok}. It is then natural to ask the following:

\begin{openproblem}
Given a universal Turing machine, is there some polynomial (or subexponential growth) entropy which is positive? 
\end{openproblem}

The constructions in this article are either Turing complete for a non-prescribed metric on a compact manifold or Turing complete for the fixed Euclidean metric on $\mathbb{R}^3$. The following question remains unanswered:
\begin{openproblem}
Does there exist a Turing complete stationary Euler flow on $(\mathbb S^3,g_{round})$ or $(\mathbb T^3,g_{flat})$?    \end{openproblem}

To finish the list of open problems, we want to emphasize that our constructions explained above yield Turing complete solutions to the stationary Euler equations, so ideal fluid flows. It is natural to ask if our techniques can be applied to construct stationary solutions to the Navier-Stokes equations that are Turing complete. Since it is impossible~\cite{BGH} to construct Turing-complete systems on compact spaces that are robust under perturbations, this means that adding viscosity to the system can destroy its computational power. This leads to the problem:  

\begin{openproblem}
Does there exist a Turing complete steady Navier-Stokes flow on some compact Riemannian $3$-manifold?    \end{openproblem}
\begin{remark}
    Analogously to the Euler equations it is possible to define Navier-Stokes on general Riemannian manifolds (where the dissipative term of the equations is modelled by the Hodge Laplacian associated to the Riemannian structure).
\end{remark}



\bibliographystyle{amsplain}

\end{document}